\documentclass{amsart}

\usepackage{graphicx}
\usepackage{amsmath}
\usepackage{amssymb}
\usepackage{mathrsfs}

\newtheorem{thm}{Theorem}[section]
\newtheorem{rem}{Remark}
\newtheorem{prop}[thm]{Proposition}
\newtheorem{lemma}[thm]{Lemma}
\newtheorem{corol}[thm]{Corollary}

\newtheorem*{thm*}{Theorem}

\newcommand{\R}{\Re}

\newcommand{\C}{{\mathbb C}}
\newcommand{\s}{\sigma_0}
\newcommand{\hf}{\tfrac{1}{2}}
\newcommand{\h}{\frac{1}{2}}
\newcommand{\ct}{{\tilde{C}}}
\newcommand{\pf}{{\mathfrak{p}}}
\newcommand{\af}{{\mathfrak{a}}}

\begin{document}

\title{Explicit Upper Bounds for L-functions on the critical line}

\author{Vorrapan Chandee}

\address{Stanford University, 450 Serra Mall, Building 380, Stanford, CA 94305}
\email{vchandee@math.stanford.edu}

\begin{abstract}
We find an explicit upper bound for general $L$-functions on the critical line, assuming the Generalized Riemann Hypothesis, and give as illustrative examples its application to some families of $L$-functions and Dedekind zeta functions. Further, this upper bound is used to obtain lower bounds beyond which all eligible integers are represented by Ramanujan's ternary form and Kaplansky's ternary forms. This improves on previous work of Ono and Soundararajan \cite{OS} on Ramanujan's form and Reinke \cite{Re} on Kaplansky's form with a substantially easier proof. 
\end{abstract}
\maketitle

\section{Introduction}
Finding upper bounds for $L$-functions on the critical line is an interesting problem in analytic number theory. One classical bound is the convexity bound, which follows from the Phragmen-Lindel\"{o}f principle and the approximate functional equation of the L-function. For any $\epsilon > 0,$ the convexity bound is 
$$ L(f,s) \ll q(f,s)^{1/4 + \epsilon}, $$ 
where $\R s =\hf$ and $q(f,s)$ is the analytic conductor of the $L$-function, whose definition we recall later. However, for many applications this convexity bound is not sufficient, and one needs to improve it for relevant $L$-functions by reducing the exponent $1/4$ by a positive number.  This is known as the sub-convexity problem.  Although such subconvexity bounds have been derived in many cases, we are still far from proving the conjectured order even for the Riemann zeta function.  The Lindel\"{o}f Hypothesis states that for any $\epsilon > 0,$
$$ L(f,s) \ll q(f,s)^{\epsilon},$$
where $\R s = \hf,$ and the implied constant depends on $\epsilon.$ This is a well known consequence of the Generalized Riemann Hypothesis (GRH) (see e.g. Corollary 5.20 p. 116 in \cite{IK}). Indeed, in the process of deducing the Lindel\"{o}f Hypothesis from GRH, we obtain that for some $A > 0$
\begin{equation} \label{eqn:lindelof}
L(f,\hf) \ll \exp\left( A \frac{\log q(f,\hf) }{\log \log q(f,\hf)} \right).
\end{equation}
In this paper, instead of giving asymptotic upper bounds for $L$-functions, we will compute an explicit upper bound for $L$-functions on the critical line, assuming GRH. This type of upper bound is useful for certain applications. For example, Ono and Soundararajan \cite{OS} used explicit upper bounds (on GRH) for $L$-functions at the critical point to show that all odd numbers larger than $2 \cdot 10^{10}$ are represented by Ramanujan's form $x^2 + y^2 + 10z^2$.  Using our work here we can give a substantially easier proof of their results. Moreover our bounds show that all odd numbers larger than $3\cdot 10^7$ are represented, so that much less computation is required. We anticipate that these results would be useful in studying representation questions for general ternary quadratic froms. We will give more details on this problem, together with other examples in \S 4.
\\
\\
To calculate the explicit upper bound, we will follow Soundararajan's technique from \cite{Moment}. The main idea of the technique is to write $\log|L(f, \hf)|$ as a short sum over prime powers and a sum over non-trivial zeros of $L(f, s)$. Furthermore in finding our upper bound for $\log|L(f, \hf)|,$ the contribution of nontrivial zeros can be ignored. The short sum over prime powers can be explicitly bounded both numerically and in term of analytic conductor, which will be defined below. We will give details of the proof in \S 2.  Before we state the main theorem and corollaries, here are basic notations that we will use throughout the paper.
\\
\\
From \cite{IK} p. 94, $L(f,s)$ is an $L$-function if it satisfies the following properties.
\\

I. It is a Dirichlet series with Euler product of degree $d \geq 1,$
\begin{equation}\label{Euler}
L(f,s) = \sum_{n \geq 1} \lambda_f(n)n^{-s} = \prod_p (1-\alpha_1(p)p^{-s})^{-1}...(1-\alpha_d(p)p^{-s})^{-1}, 
\end{equation} 
and 
\begin{equation}\label{Leqn}
-\frac{L'}{L}(f,s) = \sum_{\substack{n=p^l \\ l \geq 1}} \frac{(\alpha_1^l(p) +...+ \alpha_d^l(p))\log p}{n^{s}} = \sum_{n} \frac{a(n)}{n^s}
\end{equation}
with $\lambda_f(1) = 1, |\alpha_p| < p$ for all $p$. The series and Euler products are absolutely convergent for $\R(s) > 1.$
\\

II. Let
\begin{equation}\label{gammaEqn}
\gamma(f,s) = \pi^{-ds/2} \prod_{j=1}^d \Gamma\left(\frac{s+k_j}{2}\right)
\end{equation} 
be a gamma factor. Since we will assume GRH, $\R k_j \geq 0.$  This condition tells us that $\gamma(f,s)$ has no zero in $\C$\: and no pole for $\R(s) \geq 0.$ Let an integer $q(f) \geq 1$ be a conductor of $L(f,s)$ such that $\alpha_i(p) \neq 0$ for $p$ not dividing $q(f).$ Let 
$$\Lambda(f,s) = q(f)^{s/2}\gamma(f,s)L(f,s)$$
$L(f,s)$ must satisfy the functional equation
\begin{equation}\label{fncEqn}
\Lambda(f,s) = \epsilon(f)\overline{\Lambda}(f,1-s),
\end{equation}
where $|\epsilon(f)| = 1$, and $\overline{\Lambda}(f,s) = \overline{\Lambda(f,\overline{s})}$.
We define the {\it analytic conductor} of $L(s,f)$ to be 
\begin{equation}\label{def:cond}
C = \frac{q(f)}{\pi^d} \prod_{j=1}^d \left|\frac{1}{4} + \frac{k_j}{2} \right|.
\end{equation}
In this paper, we will assume that $\Lambda(f,s)$ has no zero or pole at $\R s = 1.$ Note that by the functional equation, this implies the same holds at at $\R s = 0.$  
\\
\\
Our main theorem will be an upper bound for $\log L(f, \hf)$ in terms of the sum over prime powers and the analytic conductor (Theorem \ref{prop:upper} in \S 2). As a corollary, we derive the Lindel\"{o}f hypothesis from the main theorem. In fact, we can find an explicit constant $A$ in (\ref{eqn:lindelof}) as below. 
\begin{corol}\label{cor:prop} Let $L(f,s)$ be an $L$-function satisfying the conditions above and $C$ be defined as in (\ref{def:cond}). Furthermore assume that $L(f,s)$ satisfies Ramanujan's conjecture. Then for $\log \log C \geq 10,$
$$ \log |L(f, \hf)| \leq \frac{23d}{25}\frac{\log C}{\log^2 \log C}+ \frac{3}{8} \frac{\log C}{\log \log C},$$
\end{corol}
\begin{rem}
Without Ramanujan's conjecture, we instead obtain
$$\log |L(f, \hf)| \leq  \frac{9}{8} \frac{\log C}{\log \log C} + O\big(\frac{d \log C}{\log^2 \log C}\big).$$
\end{rem}
\begin{rem}
It may be possible to improve the explicit constant $3/8$ appearing in Corollary \ref{cor:prop} by using a different kernel in the proof of Lemma \ref{upsigma}.
\end{rem}
Applying Corollary \ref{cor:prop} to families of $L$-functions which satisfy Ramanujan's conjecture, we then easily obtain the explicit upper bound for these $L$-functions in term of their analytic conductor. We will illustrate our results using three examples of families of $L$-functions: Dirichlet $L$-functions, holomorphic cusp forms and the Riemann zeta function with varying height on the critical line. 
\begin{corol}\label{cor:dirichlet}
Let $\chi$ be a primitive even Dirichlet character modulo $q$. Let $\ct = q\big(|t| + \hf \big).$ If $\log \log \ct \geq 10,$ then for any $t,$ 
$$|L(\hf + it, \chi)| \leq \exp\left(\frac{23}{25}\frac{\log \ct}{\log^2 \log \ct}+ \frac{3}{8} \frac{\log \ct}{\log \log \ct}\right). $$
\end{corol}

\begin{corol} \label{cor:cuspform} Let $f(s) = \sum_{n \geq 1} \lambda(n) q^n$ be a holomorphic cusp form, of weight $k \geq 1$ and level $q$, and $L(f,s) = \sum_{n \geq 1} \tfrac{\lambda(n)}{n^s}$ be an $L$-function associated to $f.$ Define $\ct = qk^2.$ If $\log \log \ct \geq 10$, then 
$$ |L(f,\hf)| \leq \exp\left(\frac{46}{25}\frac{\log \ct}{\log^2 \log \ct}+ \frac{3}{8} \frac{\log \ct}{\log \log \ct}\right). $$
\end{corol}

\begin{corol} \label{cor:zetafunction} Assume $T \leq t \leq 2T.$ Let $\ct = T$. Also assume that $\log \log \ct \geq 10$. Then
 $$|\zeta(\hf + it)| \leq \exp\left(\frac{23}{25}\frac{\log \ct}{\log^2 \log \ct}+ \frac{3}{8} \frac{\log \ct}{\log \log \ct}\right). $$
\end{corol}
\begin{rem}
The analytic conductor is a measure of the complexity of the $L$-function. For our examples above, $C \approx \ct.$ Hence we can apply Corollary \ref{cor:prop} to $\ct.$ The details will be discussed in \S 2.
\end{rem}
Finally with some modifications, we can apply Corollary \ref{prop:upper} to Dedekind zeta functions. We will prove the following Corollary in \S 3. 
\begin{corol} \label{cor:dedekind} Let $K/{\mathbb Q}$ be a number field of degree $d$. The Dedekind zeta function 
$$\zeta_K(s) = \prod_{\mathfrak{p}} (1 - (N\pf)^{-s})^{-1} $$
is an $L$-function of degree d with conductor $q$ equal to the absolute value of the discriminant of $K$. If $\log \log C \geq 10$ then 
$$|\zeta_K(\hf)| \leq  2.33\exp\left(\frac{23d}{25} \frac{\log C}{\log^2\log C} + \frac{3}{8} \frac{\log C}{\log \log C} \right).$$
\end{corol}
Acknowledgements: I am very grateful to Professor Soundararajan for his guidance throughout the making of this paper. I would like to thank Xiannan Li for helpful editorial comments. I also want to thank the referee for careful reading of this paper, and helpful comments.
\section{Proof of the Main Theorem and its Corollaries}
We will now prove an explicit upper bound for $L(f, \hf)$ in terms of a sum over powers of primes and parameters associated with $L(f,s)$.  Then we will simplify the sum and those constants to derive the corollaries stated in the introduction. Our main theorem is 
\begin{thm} \label{prop:upper}
Assume GRH and $\Lambda(f,s)$ has no pole or zero at $s=0, 1.$  Let $\lambda_0 = 0.4912...$ denote the unique positive real number satisfying $e^{-\lambda_0} = \lambda_0 + \lambda_0^2/2.$ Then for all $\frac{\log x}{2} \geq \lambda \geq \lambda_0,$ and $\log x \geq 2$ we have
\begin{eqnarray*}
|L(f,\hf)|
&\leq&  \R \sum_{n \leq x} \frac{a(n)}{n^{\hf + \tfrac{\lambda}{\log x}}\log n}\frac{\log\frac{x}{n}}{\log x} +  \left(\frac{1 + \lambda}{2} \right)\frac{\log C}{\log x} + \frac{(\lambda^2 + \lambda)d}{\log^2 x} + \frac{4de^{-\lambda}}{x^{1/2}\log^2 x}.
\end{eqnarray*}
\end{thm} 
\begin{rem} The above easily leads to an upper bound for $\log\left|L\left(f, \hf + it\right)\right|$ as well.  Indeed, we can set $L_{new}(f, s) = L(f, s+it)$, so that $L\left(f, \hf + it\right) = L_{new}\left(f, \hf\right).$ $L_{new}(f, s)$ satisfies the functional equation (\ref{fncEqn}), but with $k_j + it$ in the gamma factor in (\ref{gammaEqn}). 
\end{rem}
{\bf Proof of Theorem \ref{prop:upper}} \,\,
Let $\rho = \h + i\gamma$ run over the non-trivial zeros of $L(f,s).$ Define
$$G(s) = \R \sum_{\rho} \frac{1}{s - \rho} = \sum_{\rho} \frac{\sigma - \h}{(\sigma - \h)^2 + (\gamma - t)^2}.$$
To prove Theorem \ref{prop:upper}, we need to use the Hadamard factorization formula below. The proof can be found in \cite{IK} (see Theorem 5.6).
\begin{prop} \label{HadamardL}
Let $L(f,s)$ be an $L$-function, and $s = \sigma + it$. There exist constants $a, b$ such that 
$$\Lambda(f,s) = e^{a+bs}\prod_{\rho \neq 0,1} \left(1-\frac{s}{\rho} \right)e^{s/\rho}.$$
where $\rho$ runs over all zeros of $\Lambda(f,s)$. Hence
\begin{equation}\label{impEqn}
-\frac{L'}{L}(f,s) = \frac{1}{2}\log \frac{q}{\pi^d} + \frac{1}{2}\sum_{j=1}^d \frac{\Gamma'}{\Gamma}\left(\frac{s + k_j}{2} \right) - b - \sum_{\rho \neq 0, 1}\left( \frac{1}{s-\rho} + \frac{1}{\rho}\right)
\end{equation}
is uniformly and absolutely convergent in compact subsets which have no zeros or poles. Furthermore $\R (- b + \sum \tfrac{1}{\rho} ) = 0.$
\end{prop}
By (\ref{impEqn}), since $\R (- b + \sum \tfrac{1}{\rho} ) = 0,$ if $L(f,s) \neq 0$ we have
\begin{equation}\label{imp2}
-\R \frac{L'}{L}(f,s) =  \frac{1}{2}\log \frac{q}{\pi^d} + \frac{1}{2}\sum_{j=1}^d \R\frac{\Gamma'}{\Gamma}\left(\frac{s + k_j}{2} \right) - G(s).
\end{equation}
We need the following lemma to find upper bound for $\R\frac{\Gamma'}{\Gamma}\left(\frac{s + k_j}{2} \right).$ The proof of the lemma can be found in Appendix. 
\begin{lemma} \label{Lem:Ipart}
Let $z = x + iy,$ where $x \geq \tfrac{1}{4}.$ Then 
$$\R \frac{\Gamma'}{\Gamma} (z) \leq \log |z|. $$
\end{lemma}
From Lemma \ref{Lem:Ipart} and (\ref{imp2}), we obtain
\begin{equation} \label{eqn:imp3}
-\R \frac{L'}{L}(f,s) \leq  \frac{1}{2}\log \frac{q}{\pi^d} + \frac{1}{2}\sum_{j=1}^d \log \left|\frac{s + k_j}{2} \right| - G(s).
\end{equation}
When we integrate the inequality above as $\sigma$ varies from $\h$ to $\s = \h + \frac{\lambda}{\log x}$, we get 
\begin{eqnarray}
&& \log|L(f,\hf)| - \log|L(f,\s)| \label{halfEqn} \\
&\leq& \frac{\lambda}{2\log x} \log \frac{q}{\pi^d} + \frac{1}{2}\sum_{j=1}^d \int_{\hf}^{\s} \log \left| \frac{\sigma + k_j}{2} \right| \> d\sigma - \hf\sum_\rho \log \frac{(\s - \hf)^2 + \gamma^2}{\gamma^2} \nonumber \\
&\leq& \frac{\lambda}{2\log x} \log \frac{q}{\pi^d} + \left( \frac{\lambda}{2\log x} \right)\sum_{j=1}^d\log\left|\frac{\s + k_j}{2} \right| - \hf (\s - \hf)G(\s) \nonumber \\
&\leq&  \frac{\lambda}{2\log x} \log C + \frac{\lambda^2 d}{\log^2 x}  - \frac{\lambda}{2\log x}G(\s), \nonumber
\end{eqnarray}
where we use the fact that $\log(1+x^2) \geq \frac{x^2}{1+x^2}$ and for $\lambda \leq \tfrac{\log x}{2},$
\begin{equation} \label{eqn:logabs}
\log\left|\frac{\s + k_j}{2} \right| \leq \log\left|\frac{1}{4} + \frac{k_j}{2} \right| + \log \left|1 + \frac{\lambda}{(\log x)(\hf + k_j)} \right| \leq \log\left|\frac{1}{4} + \frac{k_j}{2} \right| + \frac{2\lambda}{\log x}.
\end{equation}
To prove Theorem \ref{prop:upper}, we need an upper bound for $\log|L(f, \s)|$, which we will obtain from Lemma \ref{upsigma} below.  This lemma is a version of Lemma 1 in \cite{Moment} for general $L$-functions. The proof of Lemma \ref{upsigma} is essentially the same as the one in \cite{Moment}; however we provide the sketch of the proof here for completeness.
\begin{lemma} \label{upsigma}
Unconditionally, for any s not conciding with 1, 0 or a zero of $L(f,s)$, and for any $x \geq 2,$ we have
\begin{eqnarray*}
-\frac{L'}{L}(f,s) &=& \sum_{n \leq x} \frac{a(n)}{n^{s}}\frac{\log\frac{x}{n}}{\log x} + \frac{1}{\log x}\left( \frac{L'}{L}(f,s) \right)' + \frac{1}{\log x} \sum_{\rho \neq 0,1} \frac{x^{\rho - s}}{(\rho - s)^2} \\
&& + \frac{1}{\log x} \sum_{j=1}^d \sum_{n=0}^{\infty} \frac{x^{-2n-k_j-s}}{(2n + k_j +s)^2},
\end{eqnarray*}
where $a(n)$ is defined in (\ref{Leqn}).
\end{lemma}
\begin{proof} With $c =$ max$(1, 2-\sigma),$ integrating term by term using the Dirichlet series expansion of $\frac{-L'}{L}(f,s)$, we have
$$\frac{1}{2\pi i} \int_{c-i\infty}^{c+i\infty} -\frac{L'}{L}(f, s+w)\frac{x^w}{w^2} \>dw = \sum_{n \leq x}\frac{a(n)}{n^{s}}\log\frac{x}{n}. $$
On the other hand, moving the line of integration to the left and calculating residues gives
\begin{eqnarray*}
\log x \left (-\frac{L'}{L}(f,s) -  \frac{1}{\log x}\left( \frac{L'}{L}(f,s) \right)' - \frac{1}{\log x} \sum_{\rho \neq 0,1} \frac{x^{\rho - s}}{(\rho - s)^2} - \frac{1}{\log x} \sum_{j=1}^d \sum_{n=0}^{\infty} \frac{x^{-2n-k_j-s}}{(2n + k_j +s)^2} \right).
\end{eqnarray*}
Equating these two expressions, we obtain the lemma. 
\end{proof}
We take $s= \sigma$ in Lemma (\ref{upsigma}), extract the real parts of both sides, and integrate over $\sigma$ from $\s$ to $\infty.$ Thus for $x \geq 2$
\begin{eqnarray*}
\log|L(f,\s)| &=& \R \big(\sum_{n \leq x} \frac{a(n)}{n^{\s}\log n}\frac{\log\frac{x}{n}}{\log x} - \frac{1}{\log x}\frac{L'}{L}(f,\s) + \frac{1}{\log x} \sum_{\rho \neq 0,1} \int_{\s}^{\infty} \frac{x^{\rho - \sigma}}{(\rho - \sigma)^2} \> d\sigma \\
&& + \frac{1}{\log x}  \int_{\s}^{\infty} \sum_{j=1}^d \sum_{n=0}^{\infty} \frac{x^{-2n-k_j-\sigma}}{(2n + k_j + \sigma)^2} \> d\sigma\big).
\end{eqnarray*}
Moreover 
$$
 \sum_\rho \left| \int_{\s}^{\infty} \frac{x^{\rho - \sigma}}{(\rho - \sigma)^2} \>d \sigma \right| \leq \sum_\rho \int_{\s}^{\infty} \frac{x^{\h-\sigma}}{|\rho - \s|^2} \>d\sigma = \sum_\rho \frac{x^{\h-\s}}{|\rho - \s|^2 \log x} = \frac{e^{-\lambda}G(\s)}{\lambda},
$$
and for $\R k_j \geq 0,$
\begin{eqnarray*}
\R \int_{\s}^{\infty} \sum_{j=1}^d \sum_{n=0}^{\infty} \frac{x^{-2n-k_j-\sigma}}{(2n + k_j + \sigma)^2} \> d\sigma 
\leq \sum_{j=1}^d \frac{x^{-\R k_j-\s}}{\log x |k_j + \s|^2} \sum_{n=0}^{\infty} x^{-2n}
\leq \frac{4de^{-\lambda}}{x^{1/2}\log x}.
\end{eqnarray*}
Hence using the previous two lines and (\ref{eqn:imp3}), we deduce that 
\begin{eqnarray} \label{sigmaEqn} 
&& \log|L(f, \s)|  \\
&\leq& \R \sum_{n \leq x} \frac{a(n)}{n^{\s}\log n}\frac{\log\frac{x}{n}}{\log x} + \frac{1}{\log x} \left( \frac{1}{2}\log \frac{q}{\pi^d} + \frac{1}{2}\sum_{j=1}^d \log \left|\frac{\s + k_j}{2} \right| - G(\s)\right) \nonumber  \\
&&  +  \frac{4de^{-\lambda}}{x^{1/2}\log^2 x} +  \frac{e^{-\lambda}G(\s)}{\lambda\log x}. \nonumber
\end{eqnarray}
Adding the inequalities (\ref{halfEqn}), (\ref{sigmaEqn}) and using (\ref{eqn:logabs}) 
we get 
\begin{eqnarray} \label{eqn:ignorenegative}
\log|L(f, \hf)|
&\leq&  \R \sum_{n \leq x} \frac{a(n)}{n^{\s}\log n}\frac{\log\frac{x}{n}}{\log x} +  \left(\frac{\lambda + 1}{2\log x} \right)\log C + \frac{(\lambda^2 + \lambda)d}{\log^2 x}  \\
&& + \frac{4de^{-\lambda}}{x^{1/2}\log^2 x} + \frac{G(\s)}{\log x}\left( \frac{e^{-\lambda}}{\lambda}- 1 - \frac{\lambda}{2} \right).  \nonumber
\end{eqnarray}
For $\lambda \geq \lambda_0$, the term involving $G(\s)$ above gives a negative contribution, and we can omit it. Hence the theorem is proved.
\\
\\
{\bf Proof of Corollary \ref{cor:prop}:} \\
Since our $L$-function $L(f, s)$ satisfies Ramanujan's conjecture, $|a(n)| \leq d\Lambda(n).$ From Theorem \ref{prop:upper} picking $\lambda = 0.5$, we have
\begin{equation}   \label{eqn:maincor}
\log |L(f, \hf)| \leq \sum_{n \leq x} \frac{d\Lambda(n)}{n^{\hf + \tfrac{1}{2\log x}}\log n}\frac{\log\frac{x}{n}}{\log x} + \frac{3}{4\log x} \log C + d\left(\frac{3}{4\log^2 x} + \frac{4e^{-0.5}}{x^{1/2}\log^2 x} \right). 
\end{equation}    
When $\log x \geq 20,  \frac{3}{4\log^2 x} + \frac{4e^{-0.5}}{x^{1/2}\log^2 x} \leq 0.0019.$ Now we need to find the upper bound of the sum over powers of primes. We shall prove the following inequality
\begin{equation} \label{eqn:sumpart}
\sum_{n \leq x} \frac{\Lambda(n)}{n^{\hf + \tfrac{\lambda}{\log x}}\log n}\frac{\log\frac{x}{n}}{\log x} \leq 3.675\frac{x^{1/2}}{\log^2 x}.
\end{equation}
\\
Let $f(t) = \frac{1}{t^{\hf + \tfrac{\lambda}{\log x}}\log t}\frac{\log\frac{x}{t}}{\log x}.$ By partial summation,  
\begin{eqnarray*}
\sum_{n \leq x} \frac{\Lambda(n)}{n^{\hf + \tfrac{\lambda}{\log x}}\log n}\frac{\log\frac{x}{n}}{\log x} &=& - \int_2^x (\sum_{n \leq t} \Lambda(n))f'(t) \> dt \\
\end{eqnarray*}
Since we assume GRH, using the result of Lemma 8 and (3.21) in \cite{RS} (pick $\delta = \frac{2}{\sqrt{2 + \sqrt{x}}}$), we obtain that for $t > 10,$
$$\sum_{n \leq t} \Lambda(n) \leq t + 0.0463\left(\frac{2\sqrt{2 + \sqrt{t}} + 2}{\sqrt{t}}\right)t$$
For $t \geq 10^5,$ we have
$$ \sum_{n \leq t} \Lambda(n) < (1.006) t.$$
The inequality is also true for $10^5 > t \geq 2$ by numerical experiment. Therefore,
\begin{eqnarray*}
\sum_{n \leq x} \frac{\Lambda(n)}{n^{\hf + \tfrac{\lambda}{\log x}}\log n}\frac{\log\frac{x}{n}}{\log x} &\leq& -1.006\int_2^x t f'(t) \> dt  \\
&=& 1.006\left(\frac{2\log\tfrac{x}{2}}{2^{1/2 + 1/(2\log x)}\log 2\log x} + \int_2^x f(t) \> dt \right)
\end{eqnarray*}
We change variable $t = x/y$ and obtain
$$\int_2^x f(t) \> dt = \frac{e^{-\lambda}x^{1/2}}{\log^2 x} \int_1^{x/2} \frac{\log y}{y^{3/2- 1/(2\log x)}(1 - \tfrac{\log y}{\log x})} \> dy \leq 5.961\frac{e^{-\lambda}x^{1/2}}{\log^2 x},$$
where the last inequality follows because the second integral above is a decreasing function of $x$ when $x>10^4$.  The constant appearing on the right hand side is derived by substituing $x = e^{20}$ in the integral. Finally by absorbing the constant term into $\frac{x^{1/2}}{\log^2 x}$ term, we derive (\ref{eqn:sumpart}). Choosing $x = \log^2 C $ in (\ref{eqn:maincor}) and applying (\ref{eqn:sumpart}), we prove the corollary. 
\\
\\
{\bf Proof of Corollary \ref{cor:dirichlet}:}  The Dirichlet $L$-function $L(s + it,\chi)$ is an $L$-function of degree 1 with conductor $q$ and satisfies Ramanujan's conjecture. The gamma factor is $\gamma(s) = \pi^{-s/2}\Gamma \left(\tfrac{s + it}{2} \right),$ so the analytic conductor of $L(s + it, \chi)$ as defined in (\ref{def:cond}) is $C = \frac{q}{2\pi} \sqrt{\tfrac{1}{4} + t^2}.$ For any $t$, $\log C  \leq \log q\big(\hf + |t| \big).$ 
Therefore Corollary \ref{cor:prop} can be applied to $\ct = q\big(\hf + |t| \big),$ and we then obtain the corollary.   
\\
\\
{\bf Proof of Corollary \ref{cor:cuspform}:} $L(f,s)$ is an $L$-function of degree 2 with conductor q and gamma factor
$$\gamma(f,s) = \pi^{-s} \Gamma\left(\frac{s + (k-1)/2}{2} \right)\Gamma\left(\frac{s + (k+1)/2}{2} \right). $$
$L(f,s)$ satisfies Ramanujan's conjecture and its analytic conductor $C$ is $\tfrac{q}{\pi^2}\left(\tfrac{k}{2} - \tfrac{1}{4} \right)\left(\tfrac{k}{2} + \tfrac{3}{4} \right).$ For $k \geq 1$, we have
$$\log C = \log qk^2 - \log 4\pi^2 + \log \big(1 + \tfrac{1}{k} - \tfrac{3}{4k^2}  \big) \leq \log \ct - \log 4\pi^2 +  \tfrac{1}{k} - \tfrac{3}{4k^2}  \leq \log \ct.$$
Therefore we can apply Corollary \ref{cor:prop} to $\ct = qk^2$ and get the result. 
\\
\\
{\bf Proof of Corollary \ref{cor:zetafunction}} The Riemann zeta function $\zeta(s + it)$ is an $L$-function of degree 1. The gamma factor is $\gamma(s) = \pi^{-s/2}\Gamma\big(\tfrac{s + it}{2} \big).$  We can still apply Corollary \ref{cor:prop}, but we need to add a constant term derived from the pole $s = 1 - it$. 
The analytic conductor $C$ is $\tfrac{1}{2\pi}\sqrt{\tfrac{1}{4} + t^2}.$ For $T \leq t \leq 2T$ and $e^{20} \leq (\log \ct)^2 \leq T^2,$ 
$$\log \ct \leq \log t + \log \frac{1}{2\pi} + \frac{1}{8t^2} \leq  \log T + \log \frac{1}{\pi} + \frac{1}{8t^2} \leq \log T.$$
Therefore we can apply Corollary $\ref{cor:prop}$ to $\ct = T.$ 
Now we consider (\ref{impEqn}). For $\zeta(s + it)$ it becomes 
$$-\frac{\zeta '}{\zeta}(s + it) = \frac{1}{2}\log \frac{q}{\pi^d} + \frac{1}{2}\sum_{j=1}^d \frac{\Gamma'}{\Gamma}\left(\frac{s + k_j}{2} \right) - b - \sum_{\rho \neq 0, 1}\left( \frac{1}{s-\rho} + \frac{1}{\rho}\right) + \frac{1}{s - 1 + it} + \frac{1}{s + it}.$$
Following the proof of Theorem \ref{prop:upper}, we obtain that the contribution of the pole terms is bounded by $1.4 \cdot 10^{-6},$ which is negligible when we apply Corollary \ref{cor:prop}.  

\section{Dedekind Zeta Functions}
Our upper bound can also be applied to Dedekind zeta functions. In this section, we will prove Corollary \ref{cor:dedekind}.
\begin{proof} The gamma factor of $\zeta_K(s)$ is 
$$\gamma(s) = \pi^{-ds/2}\Gamma\left(\frac{s}{2} \right)^{r_1 + r_2} \Gamma\left(\frac{s + 1}{2} \right)^{r_2}  $$
where $r_1$ is the number of real embeddings of $K$ and $r_2$ the number of pairs of complex embeddings, so that $d = r_1 + 2r_2.$ 

Let $H(s) = \zeta_K(s) (s-1).$ Since the Dedekind zeta function has a simple pole at $s = 1,$ $H(s)$ is entire. The equation (\ref{imp2}) for $H(s)$ is 
$$-\R \frac{H'}{H}(s) =  \frac{1}{2}\log \frac{q}{\pi^d} + \frac{r_1 + r_2}{2}\R \frac{\Gamma'}{\Gamma}\left( \frac{s}{2} \right) + \frac{r_2}{2}\R \frac{\Gamma'}{\Gamma}\left( \frac{s + 1}{2} \right) + \R\frac{1}{s} -  G(s).$$
Similar to (\ref{halfEqn}), we integrate the inequality above from as $\sigma$ varies from $\hf$ to $\sigma_0$, use Lemma \ref{Lem:Ipart}, and obtain
\begin{equation} \label{eqn:half2}
\log|H(\hf)| - \log|H(\s)| \leq \frac{\log C}{4\log x}  + \frac{d}{4 \log^2 x} - \hf(\s - \hf)G(\s) + \frac{1}{\log x}.
\end{equation}
Now we need an upper bound for $\log|H(\s)|,$ which will be derived in the same way as the upper bound for $\log|L(f,\s)|.$ 
\begin{lemma} \label{lem:h}
Unconditionally, for any positive real number $t > \hf$ and any $x \geq 2$ we have
\begin{eqnarray*}
- \frac{H'}{H}(t) &\leq&  \sum_{n \leq x} \frac{d\Lambda(n)}{n^t} \frac{\log \tfrac{x}{n}}{\log x} + \frac{1}{\log x}\left( \frac{H'}{H}(t) \right)' + \frac{1}{\log x} \sum_{\rho \neq 0,1} \frac{x^{\rho - t}}{(\rho - t)^2} \\
&& + \frac{(r_1 + r_2)}{\log x} \sum_{n=0}^{\infty} \frac{x^{-2n-t}}{(2n + t)^2}  + \frac{r_2}{\log x}  \sum_{n=0}^{\infty} \frac{x^{-2n - 1 -t}}{(2n + 1 + t)^2}. 
\end{eqnarray*}
\end{lemma} 
\begin{proof}
By definition of $H(s)$ and by using the same arguments as in Lemma \ref{upsigma}, we have
\begin{eqnarray*}
&&- \frac{H'}{H}(t) =  -\frac{\zeta_K'}{\zeta_K}(t) - \frac{1}{t - 1} \\
&\leq& \sum_{N \af \leq x} \frac{\Lambda_K(\af)}{(N\af)^t} \frac{\log\tfrac{x}{N\af}}{\log x}  + \frac{1}{\log x}\left( \frac{H'}{H}(t) \right)' + \frac{1}{(t-1)^2 \log x} + \frac{1}{\log x} \sum_{\rho \neq 0,1} \frac{x^{\rho - t}}{(\rho - t)^2} \\
&& + \frac{(r_1 + r_2)}{\log x} \sum_{n=0}^{\infty} \frac{x^{-2n-t}}{(2n + t)^2}   + \frac{r_2}{\log x} \sum_{n=0}^{\infty} \frac{x^{-2n - 1 -t}}{(2n + 1 + t)^2} - \frac{x^{-t}}{t^2\log x} - \frac{x^{1-t}}{(t-1)^2\log x} - \frac{1}{t - 1}.
\end{eqnarray*}
Furthermore 
$$\sum_{N \af \leq x} \frac{\Lambda_K(\af)}{(N\af)^t} \frac{\log\tfrac{x}{N\af}}{\log x} \leq  \sum_{n \leq x} \frac{d\Lambda(n)}{n^t} \frac{\log \tfrac{x}{n}}{\log x},$$ 
and $\frac{1}{(s-1)^2 \log x} - \frac{x^{1-s}}{(s-1)^2\log x} - \frac{1}{s - 1}$ and $- \frac{x^{-s}}{s^2\log x}$ are entire when $\R s \geq 0$ and are less than or equal to zero for any positive real number $t$ and $x \geq 2.$
Combining these facts and the inequalities above, we prove the lemma.
\end{proof}
We take $t = \sigma$ in Lemma \ref{lem:h} and integrate over $t$ from $\s$ to $\infty.$ By the same reasoning as in the proof of Theorem \ref{prop:upper},
\begin{eqnarray*}
\log|H(\hf)| &\leq& \sum_{n \leq x} \frac{d\Lambda(n)}{n^{\hf + \tfrac{1}{2\log x}}\log n} \frac{\log \tfrac{x}{n}}{\log x} + \frac{3\log C}{4\log x} + \frac{3d}{4\log^2 x} + \frac{4de^{-0.5}}{x^{1/2}\log^2 x} + \frac{2}{\log x + 1}  + \frac{1}{\log x}.
\end{eqnarray*}
From Corollary \ref{cor:prop} (let $x = \log^2 C$), we obtain 
$$ \log |\zeta_K(\hf)| \leq  \frac{23d}{25}\frac{\log C}{\log^2\log C} + \frac{3}{8} \frac{\log C}{\log \log C} + \frac{3}{2\log \log C} + \log 2.$$
The lemma follows because $\log \log C \geq 10.$
\end{proof}
\section{Application to Positive Definite Ternary Quadratic Forms}
We call $N$ {\it eligible} for a positive ternary quadratic form $f(x,y,z)$ if there are no congruence obstructions prohibiting $f$ from representing $N$. Let $\chi = \left(\frac{-40N}{\cdot}\right)$ be the Kronecker-Legendre symbol. Also define
$$ L(s,\chi) := \sum_{n=1}^{\infty} \frac{\chi(n)}{n^s}, \,\,\,\,\,\,\,\,\,
{\textrm{and}} \,\,\,\,\,\,\,\,\,
L(E(-10N),s) := \sum_{n=1}^{\infty} \frac{A(n)\chi(n)}{n^s}. 
$$
By Dirichlet's class number formula (see \cite{Da}) and special case of Waldspurger's theorem connecting the Fourier coefficients of half-integral weight cusp forms and $L(E(-10N),1)$ (Theorem 2 of \cite{OS}), Ono and Soundararajan showed that if $N$ is an eligible square-free integer coprime to 10 and is not represented by Ramanujan's form, then 
\begin{equation}\label{lbound}
\frac{L(E(-10N),1)}{L(1, \chi)^2} \geq \frac{2}{7} \left(\frac{q}{4\pi^2}\right)^{1/4},
\end{equation} 
where $q$ is a conductor of $E(-10N),$ and its value is $1600N^2$ (see Prop.2 in \cite{OS}). Ono and Soundarajan proved that (\ref{lbound}) failed when $N \geq 2 \cdot 10^{10}$. Note that for $N \leq 2 \cdot 10^{10}$, W.Galway verified by computer whether it is represented by the form. The difficulty of showing this is finding an upper bound for $L(E(-10N),1)$ because of the big contribution of zeros on the critical line. It involves long and complicated calculations. 
\\
\\
As seen in the proof of Theorem \ref{prop:upper}, the technique from \cite{Moment} allows us to ignore the contribution of nontrivial zeros. Therefore by applying upper bound in Theorem \ref{prop:upper}, not only is our calcuation much simpler but also we get a better lower bound for $N$, i.e. it requires $N \geq 3 \cdot 10^7$ to yield contradiction for (\ref{lbound}). 
\\
\\
With Ono and Soundararajan's methods of deriving (\ref{lbound}) and the upper bound in Theorem \ref{prop:upper}, under GRH, we may be able to obtain not too large positive integer $N$ such that if $m \geq N$ and is represented by the spinor genus of a positive definite ternary form, then it is represented by the form itself. For example, in this paper, we will apply those techniques to Kaplansky's form and show that if $N \geq 2 \cdot 10^8$ is a squarefree integer, then it is represented by Kaplansky's forms: $\varphi_1(x,y,z) = x^2 + y^2 + 7z^2$ and $\varphi_2(x,y,z) = x^2 + 2y^2 + 2yz + 4z^2$. This result will have a simpler proof and a better lower bound than the one that T.Reinke \cite{Re} derived. 
\\
\\
T. Reinke \cite{Re} applied the method in \cite{OS} to prove the analog of (\ref{lbound}). If $N $ is an eligible square-free integer and is not represented by $\varphi_j$, where $j = 1, 2,$ then 
\begin{equation} \label{lboundkap}
\frac{L(E(-7N),1))}{L(1,\chi_d)^2} \geq  
\left\{ \begin{array}{l}
\frac{34}{101}\frac{q_N}{4\pi^2} \,\,\,\,\,\,\,\,\,\,\,\, {\rm if} \,\,\, (N,7) = 1 \\
\frac{41}{101}\frac{q_N}{4\pi^2} \,\,\,\,\,\,\,\,\,\,\,\, {\rm if} \,\,\, (N,7) > 1,
\end{array} \right.
\end{equation}
where $q_N = 28^2N^2$ for $(N,7) = 1,$ and $q_N = 7 \cdot \left( \frac{4}{7} \right)^2 N^2$ for $(N,7) > 1.$ The inequality (\ref{lboundkap}) fails when $N \geq 10^{12}.$
\\
\\
Since the proof for both Ramanujan's form and Kaplansky's forms will be the same, for simplicity of notation let $L_E(s)$ be either $L(E(-10N),s)$ or $L(E(-7N),s),$ and $L(s,\chi)$ be either $L(s, \chi_{-40N})$ or $L(s, \chi_d),$ where $d = -28N$ when $(N,7) = 1$ and $d = -\frac{4N}{7}$ when $(N,7) > 1.$ We will get upper bounds for $L_E(1)$ and lower bounds for $L(1, \chi).$
\\
\\
Let $L_n(s) = L_E(s + \hf).$ From functional equation of $L_E(s),$ we have
$\Lambda(s) = \Lambda(1-s),$ where
$$\Lambda(s) = q^{\tfrac{s}{2}} \pi^{-s} \Gamma\left(\tfrac{s + \hf}{2}\right)\Gamma\left(\tfrac{s + \tfrac{3}{2}}{2}\right)L_n(s). $$
Since $L_E(s)$ is $L$-function of degree 2 and 
$$ -\frac{L'_E}{L_E}(s) = \sum_{n=1}^{\infty} \frac{\lambda(n)\chi(n)}{n^s},$$
by Theorem \ref{prop:upper} choosing $\lambda = 0.5$ we obtain
\begin{eqnarray} \label{upLaone} 
\log|L_E(1)| &=& \log|L_n(\hf)| \\ 
&\leq& \R \sum_{n \leq x} \frac{\lambda(n)\chi(n)}{n^{1 + \frac{0.5}{\log x}}\log n}\frac{\log \left(\frac{x}{n} \right)}{\log x} + \frac{3}{4\log x} \log \frac{q}{4\pi^2} +\frac{3\log 2}{4\log x} \nonumber \\ 
&& + \frac{3}{2\log^2 x} + \frac{8e^{-0.5}}{x^{1/2}\log^2 x}. \nonumber
\end{eqnarray} 
For a lower bound for $\log|L(1, \chi)|$, we prove the following proposition.
\begin{prop}\label{lowLone} Assume GRH. Let $\chi$ be a primitive real character mod q, $y \geq 2$, $a(y) = \frac{1}{\log y} + \frac{2}{\sqrt{y} \log^2 y},$ and $b(y) = \frac{1}{1-\frac{2}{\sqrt{y}\log y} - \frac{2}{\log y}}. $ We have
\begin{eqnarray*}
&&\log|L(1, \chi)| \\ 
&\geq&  \sum_{n \leq y} \frac{\Lambda(n)\chi(n)}{n\log n} \frac{\log \left(\frac{y}{n} \right)}{\log y}  + \frac{1}{4\log y} \log \frac{q}{4\pi^2} + \frac{\log 4}{4\log y} -\frac{\gamma}{2\log y} \\ 
&& + \,\, a(y)b(y) \left( \sum_{n \leq y} \frac{\Lambda(n)\chi(n)}{n} \frac{\log \left(\frac{y}{n} \right)}{\log y} - \frac{1}{4}\log \frac{q}{4\pi^2}  -  \frac{\pi^2}{24\log y} -  \frac{\log 4}{4} + \frac{\gamma}{2}  \right).
\end{eqnarray*}
\end{prop}
\begin{proof} Let $\rho = \frac{1}{2} + i\gamma$ run over the non-trivial zero of $L(s, \chi).$ Define 
$$F(s) = \R \sum_\rho \frac{1}{s-\rho} = \sum_\rho \frac{\sigma - 1/2}{(\sigma-1)^2 + (t-\gamma)^2}. $$
From (17) and (18) of Davenport [Da], Chapter 12, if $L(s,\chi) \neq 0,$
\begin{equation}\label{fnL}
-\R \frac{L'}{L}(s, \chi) = \frac{1}{4} \log \frac{q}{\pi^2} + \frac{1}{2}\R \frac{\Gamma'}{\Gamma}\left(\frac{s+1}{2}\right) - F(s).
\end{equation}  
\\
For $s \geq 1$, the contribution from the trivial zeros $\sum_{k=0}^{\infty} \frac{y^{-2k -1 -s}}{(2k + 1 +s)^2} \geq 0$, and so by the same arguments as in Lemma (\ref{upsigma}), we have
\begin{equation}\label{fnL2}
- \R \frac{L'}{L}(s, \chi) \geq \sum_{n \leq y} \frac{\Lambda(n)\chi(n)}{n^s} \frac{\log \left(\frac{y}{n} \right)}{\log y} + \frac{1}{\log y}\left( \frac{L'}{L}(s,\chi) \right)' + \frac{1}{\log y} \R \sum_\rho \frac{y^{\rho - s}}{(\rho - s)^2}.
\end{equation}
Integrating (\ref{fnL}) as $\sigma$ varies from 1 to infinity, using (\ref{fnL}) and observing that
$$ \R \sum_\rho  \int_{1}^{\infty} \frac{y^{\rho - \sigma}}{(\rho - \sigma)^2} \>d \sigma  \geq -\sum_\rho \int_{1}^{\infty} \frac{y^{1/2-\sigma}}{|\rho - 1|^2} \>d\sigma = -\sum_\rho \frac{y^{-1/2}}{|\rho - 1|^2 \log y} = -\frac{2y^{-1/2}F(1)}{ \log y}, $$
we obtain
\begin{equation}\label{logL1}
\log|L(1,\chi)| \geq  \sum_{n \leq y} \frac{\Lambda(n)\chi(n)}{n\log n} \frac{\log \left(\frac{y}{n} \right)}{\log y}  + \frac{1}{4\log y} \log \frac{q}{4\pi^2} + \frac{\log 4}{4\log y} -\frac{\gamma}{2\log y}  - a(y)F(1),
\end{equation}
where  $ \frac{\Gamma'}{\Gamma} (1) = -\gamma.$
\\
\\
To prove the proposition, we need to find the lower bound for $-F(1)$. First note that taking derivative of (\ref{fnL}) at $s = 1$, we get
\begin{equation}\label{bder}
\R\left(\frac{L'}{L}(1,\chi)\right)' = -\frac{1}{4}\left(\frac{\Gamma'}{\Gamma}\right)'(1) - \R\sum_{\rho} \frac{1}{(1-\rho)^2} \geq - \frac{\pi^2}{24} - 2F(1).
\end{equation}
From (\ref{fnL}),(\ref{fnL2}) at $s=1$ and (\ref{bder}), 
\begin{eqnarray*}
-F(1)
&\geq&    \sum_{n \leq y} \frac{\Lambda(n)\chi(n)}{n} \frac{\log \left(\frac{y}{n} \right)}{\log y} - \frac{\pi^2}{24\log y}-\frac{2F(1)}{\log y} + \R \frac{1}{\log y} \sum_\rho \frac{y^{\rho - 1}}{(\rho - 1)^2} - \frac{1}{4} \log \frac{q}{\pi^2} + \frac{\gamma}{2} \\
&\geq&  \sum_{n \leq y} \frac{\Lambda(n)\chi(n)}{n} \frac{\log \left(\frac{y}{n} \right)}{\log y} - \frac{\pi^2}{24\log y}-\frac{2F(1)}{\log y} - \frac{2F(1)}{\sqrt{y}\log y}  - \frac{1}{4} \log \frac{q}{4\pi^2} - \frac{\log 4}{4} + \frac{\gamma}{2}. 
\end{eqnarray*}
Therefore,
$$ -F(1) \geq b(y) \left(\sum_{n \leq y} \frac{\Lambda(n)\chi(n)}{n} \frac{\log \left(\frac{y}{n} \right)}{\log y} - \frac{\pi^2}{24\log y} - \frac{1}{4} \log \frac{q}{4\pi^2} - \frac{\log 4}{4} + \frac{\gamma}{2}  \right). $$
Putting the inequality above into (\ref{logL1}), we prove the proposition.
\end{proof}
To get an upper bound for $ \log\frac{|L_E(1)|}{|L(1,\chi)|^2}$, we choose $x = 600$ in (\ref{upLaone}) and choose $y = 2100 $ in Proposition \ref{lowLone}.   Thus
\begin{eqnarray*}
\log\frac{|L_E(1)|}{|L(1)|^2}
&\leq& 0.147695 + 0.14158\log \frac{q}{4\pi^2}+  \sum_{\substack{n = p^k \\ p \leq 600 }} \big(\frac{\lambda(n)}{n^{0.5/\log{x}}\log n}\frac{\log \left(\frac{x}{n} \right)}{\log x} - \frac{2\Lambda(n)}{\log n}\frac{\log \left(\frac{y}{n} \right)}{\log y} \\
&& - 2a(y)b(y) \Lambda(n)\frac{\log \left(\frac{y}{n} \right)}{\log y} \big)\frac{\chi(n)}{n} +  \sum_{n = p = 601}^{2100} \left|\frac{2}{\log n} + 2a(y)b(y)\right|\frac{\Lambda(n)}{n}\frac{\log \left(\frac{y}{n} \right)}{\log y}.
\end{eqnarray*}
To calculate the first sum, we use the exact value of $\lambda(n)$ of $L_E(s)$ (see \cite{OS} section 2 for Ramanujan's form and \cite{Re} section 3.2 for Kaplansky's form). Using computer, for Ramanujan's form, we get contradiction for (\ref{lbound}) when $N \geq 3 \cdot 10^7.$ Similarly, for Kaplansky's forms, (\ref{lboundkap}) fails when $N \geq 2\cdot 10^8.$

\section{Appendix: Proof of Lemma \ref{Lem:Ipart}}
\begin{proof} From \cite{A} p.202, we have
$$\R  \frac{\Gamma'}{\Gamma} (z) = \log |z| - \R \frac{1}{2z} + I(z), $$
where 
$$I(z) =  - \R\int_0^{\infty} \frac{2\eta}{\eta^2 + z^2}\cdot \frac{d\eta}{e^{2\pi\eta} - 1} = - \int_0^{\infty} \frac{2\eta(\eta^2 + x^2 - y^2)}{(\eta^2 + x^2 - y^2)^2 + 4x^2y^2} \frac{d\eta}{e^{2\pi\eta}- 1}. $$ 
If $y^2 - x^2 \leq 0,$ it is clear that $I(z) \leq 0,$ and $\R \frac{1}{2z} \geq 0;$ hence the lemma is proved. Now we assume $y^2 - x^2 > 0.$ There are two possiblities for this case. \\
\\
{\bf Case 1:} $|z| \geq 4.$ 
\\
To prove the lemma, first we will show that
\begin{equation} \label{eqn:ipart}
I(z) \leq \frac{15}{128|z|^2} + \frac{4e^{-2\pi |z|/3}}{3}\left(  1 + \frac{1}{|z|^2} \right).
\end{equation}
Integrating $I(z)$ by part we obtain
\begin{eqnarray*}
- \R\int_0^{\infty} \frac{2\eta}{\eta^2 + z^2}\cdot \frac{d\eta}{e^{2\pi\eta} - 1} &=& \frac{1}{\pi} \R \int_0^{\infty} \frac{z^2 - \eta^2}{(\eta^2 + z^2)^2} \log (1-e^{-2\pi\eta}) \> d\eta \\
&\leq&  \frac{1}{\pi} \int_0^{|z|/3} + \int_{|z|/3}^{\infty} \frac{|z^2 - \eta^2|}{|\eta^2 + z^2|^2} \log \left(\frac{1}{1-e^{-2\pi\eta}}\right) \> d\eta. 
\end{eqnarray*}
In the first integral $|z^2 - \eta^2| \leq \tfrac{10}{9}|z|^2,$ $|\eta^2 + z^2| \geq \tfrac{8}{9}|z|^2,$ and 
$\int_{|a|}^{\infty}\log \left(\frac{1}{1-e^{-2\pi\eta}}\right) \> d\eta \leq \frac{\pi}{12e^{2\pi|a|}}.$ Hence
$$ \frac{1}{\pi}\int_0^{|z|/3} \frac{|z^2 - \eta^2|}{|\eta^2 + z^2|^2} \log \left(\frac{1}{1-e^{-2\pi\eta}}\right) \> d\eta \leq \frac{15}{128|z|^2}.$$
In the second integral $\frac{|z^2 - \eta^2|}{|\eta^2 + z^2|^2} \leq \frac{1}{|\eta^2 + z^2|^2} + \frac{|z|^2}{|\eta^2 + z^2|^2}.$ Also $|\eta^2 + z^2| = |z - i\eta|\cdot |z + i\eta| \geq \tfrac{|z|}{4}.$ Therefore 
$$\frac{1}{\pi}\int_{|z|/3}^{\infty} \frac{|z^2 - \eta^2|}{|\eta^2 + z^2|^2} \log \left(\frac{1}{1-e^{-2\pi\eta}}\right) \> d\eta \leq \frac{4e^{-2\pi |z|/3}}{3}\left(  1 + \frac{1}{|z|^2} \right). $$
Combining the inequalities above, we proved (\ref{eqn:ipart}). Since $x \geq 1/4$ and $|z| \geq 4,$ by (\ref{eqn:ipart}), 
$$-\R \frac{1}{2z} + I(z) \leq -\frac{1}{8|z|^2} + \frac{15}{128|z|^2} + \frac{4e^{-2\pi |z|/3}}{3}\left(  1 + \frac{1}{|z|^2} \right) \leq 0.$$
\\
{\bf Case 2:} $|z| < 4$. \\
From the last expression of (\ref{eqn:ipart}), the integrand of $I(z)$ is greater than or equal to zero when $\eta \leq \sqrt{y^2 - x^2}.$ Hence
$$-\R \frac{1}{2z} + I(z) \leq -\frac{x}{2(x^2 + y^2)} - \int_0^{\sqrt{y^2 - x^2}} \frac{2\eta(\eta^2 + x^2 - y^2)}{(\eta^2 + x^2 - y^2)^2 + 4x^2y^2} \frac{d\eta}{e^{2\pi\eta}- 1}. $$
Let $f(x,y)$ be the right hand side of the inequality above. Once we show that $f(x,y) \leq 0,$ the lemma will be proved. Without loss of generality, we can assume that $y \geq 0.$ For any fixed $y > x,$ 
$$\frac{\partial}{\partial x} f(x,y) = \frac{2x^2 - 2y^2}{4(x^2 + y^2)^2} + \int_0^{\sqrt{y^2 - x^2}} \frac{2\eta}{e^{2\pi\eta}- 1}\frac{2x(\eta^2 + x^2 - y^2)^2 + 8xy^2(\eta^2 -y^2)}{((\eta^2 + x^2 - y^2)^2 + 4x^2y^2)^2} \> d\eta. $$
Since $y^2 - x^2 > 0,$ the first term is less than 0. Also $2x(\eta^2 + x^2 - y^2)^2 + 8xy^2(\eta^2 -y^2) < 0$ because $2y\sqrt{y^2 - \eta^2} > y^2 - \eta^2 \geq y^2 - x^2 -\eta^2 .$ Therefore $f(x,y)$ is decreasing with respect to $x$. Because $x \geq 1/4$, then for any fixed $y > x,$ $f(x,y) \leq f(1/4, y).$ We know that $|z| < 4.$ From numerical computation, $f(1/4,y) \leq 0$ for $1/4 \leq y \leq 4.$ 
\end{proof}

\end{document}